\documentclass[12pt]{amsart}
\usepackage{graphicx}
\usepackage[latin2]{inputenc}
\usepackage{amsmath,amssymb,amsbsy,amsthm,amsfonts}
\usepackage{epsfig} 

\usepackage{soul}

\begin{document}
\numberwithin{equation}{section} 
\def\note#1{\marginpar{\small #1}}

\def\tens#1{\pmb{\mathsf{#1}}}
\def\vec#1{\boldsymbol{#1}}

\def\norm#1{\left|\!\left| #1 \right|\!\right|}
\def\fnorm#1{|\!| #1 |\!|}
\def\abs#1{\left| #1 \right|}
\def\ti{\text{I}}
\def\tii{\text{I\!I}}
\def\tiii{\text{I\!I\!I}}

\def\pard#1{\partial_{#1}}

\def\diver{\mathop{\mathrm{div}}\nolimits}
\def\grad{\mathop{\mathrm{grad}}\nolimits}
\def\Div{\mathop{\mathrm{Div}}\nolimits}
\def\Grad{\mathop{\mathrm{Grad}}\nolimits}

\def\tr{\mathop{\mathrm{tr}}\nolimits}
\def\cof{\mathop{\mathrm{cof}}\nolimits}
\def\det{\mathop{\mathrm{det}}\nolimits}

\def\lin{\mathop{\mathrm{span}}\nolimits}
\def\pr{\noindent \textbf{Proof: }}
\def\pp#1#2{\frac{\partial #1}{\partial #2}}
\def\dd#1#2{\frac{\d #1}{\d #2}}

\def\T{\mathcal{T}}
\def\R{\mathcal{R}}
\def\Re{\mathbb{R}}
\def\bx{\vec{x}}
\def\be{\vec{e}}
\def\bef{\vec{f}}
\def\bb{\vec{b}}
\def\bec{\vec{c}}
\def\bs{\vec{s}}
\def\ba{\vec{a}}
\def\bn{\vec{n}}
\def\bphi{\vec{\varphi}}
\def\btau{\vec{\tau}}
\def\bc{\vec{c}}
\def\bg{\vec{g}}

\def\bW{\tens{W}}
\def\bA{\tens{A}}
\def\bT{\tens{T}}
\def\bD{\tens{D}}
\def\bF{\tens{F}}
\def\bB{\tens{B}}
\def\bC{\tens{C}}
\def\bV{\tens{V}}
\def\bS{\tens{S}}
\def\bI{\tens{I}}
\def\bi{\vec{i}}
\def\bv{\vec{v}}
\def\bfi{\vec{\varphi}}
\def\bk{\vec{k}}
\def\b0{\vec{0}}
\def\bom{\vec{\omega}}
\def\bw{\vec{\omega}}
\def\bz{\vec{z}}
\def\p{\pi}
\def\bu{\vec{u}}
\def\bq{\vec{q}}

\def\ID{\mathcal{I}_{\bD}}
\def\IP{\mathcal{I}_{p}}
\def\Pn{(\mathcal{P})}
\def\Pe{(\mathcal{P}^{\eta})}
\def\Pee{(\mathcal{P}^{\varepsilon, \eta})}
\def\dx{\; \mathrm{d}x}
\def\dt{\; \mathrm{d}t}

\def\Ln#1{L^{#1}_{\bn}}

\def\Wn#1{W^{1,#1}_{\bn}}

\def\Lnd#1{L^{#1}_{\bn, \diver}}

\def\Wnd#1{W^{1,#1}_{\bn, \diver}}

\def\Wndm#1{W^{-1,#1}_{\bn, \diver}}

\def\Wnm#1{W^{-1,#1}_{\bn}}

\def\Lb#1{L^{#1}(\partial \Omega)}

\def\Lnt#1{L^{#1}_{\bn, \btau}}

\def\Wnt#1{W^{1,#1}_{\bn, \btau}}

\def\Lnd#1{L^{#1}_{\bn, \btau, \diver}}

\def\Wntd#1{W^{1,#1}_{\bn, \btau, \diver}}

\def\Wntdm#1{W^{-1,#1}_{\bn,\btau, \diver}}

\def\Wntm#1{W^{-1,#1}_{\bn, \btau}}

\newtheorem{Theorem}{Theorem}[section]
\newtheorem{Example}{Example}[section]
\newtheorem{Lemma}{Lemma}[section]
\newtheorem{Remark}{Remark}[section]
\newtheorem{Definition}{Definition}[section]
\newtheorem{Corollary}{Corollary}[section]


\title[Mean field dependent dynamics]{Bellman systems with mean field dependent dynamics}\thanks{A. Bensoussan acknowledges the support of National Science Foundation grants DMS 1303775 and 1612880 and the Hong Kong SAR Research Grant  Council GRF 500113 and 11303316.  Miroslav Bul\'{\i}\v{c}ek thanks
to Hausdorff center for mathematics at Unviversity of Bonn and to the Czech Science Foundation (grant no. 16-03230S)}

\author[A. Bensoussan]{Alain Bensoussan}
\address{Jindal School of Management, International Center for Decision and Risk Analysis\\
University of Texas at Dallas\\
800 W. Campbell Rd, SM30
Richardson, TX 75080-3021, USA
}

\address{Department SEEM\\
City University of Hong Kong\\
Hung Hom, Kowloon, Hong Kong}
\email{axb046100@utdallas.edu}

\author[M. Bul\'{\i}\v{c}ek]{Miroslav Bul\'{\i}\v{c}ek}

\address{Mathematical Institute, Faculty of Mathematics and Physics, Charles University\\ Sokolovsk\'{a} 83,
186 75 Praha 8, Czech Republic}
\email{mbul8060@karlin.mff.cuni.cz}

\author[J. Frehse]{Jens Frehse}

\address{Institute for Applied Mathematics, Department of Applied Analysis \\
Endenicher Allee 60, 53115 Bonn, Germany}
\email{erdbeere@iam.uni-bonn.de}

 \keywords{Stochastic games; Bellman equation; mean filed equation; nonlinear elliptic
equations; weak solution; maximum principle}
 \subjclass[2000]{35J60, 35K55, 35J55, 35B65}

\begin{abstract}
We deal with nonlinear elliptic and parabolic systems that are the Bellman like systems associated to stochastic differential games with mean field dependent dynamics. The key novelty of the paper is that we allow heavily mean field dependent dynamics. This in particular leads to a system of PDE's with critical growth, for which it is rare to have an existence and/or regularity result. In the paper, we introduce a structural assumptions that cover many cases in stochastic differential games with mean filed dependent dynamics for which we are able to establish the existence of a weak solution. In addition, we present here a completely new method for obtaining the maximum/minimum principles for systems with critical growths, which is a starting point for further existence and also qualitative analysis.
\end{abstract}

\dedicatory{Paper dedicated to Professor Philippe~G.~Ciarlet}

\maketitle

\section{Introduction}

In recent literature, mainly scalar Bellman equations which are coupled with a Fokker--Planck equation are studied ad we refer to starting paper \cite{LasLi06} or to a survey \cite{GoPi15}, se also \cite{Gogo}. These equations model a Nash game with a large number of players behaving similarly, so that the decision can be approximate by a single decision make (a representative agent). So we have a scalar Bellman equation. The present paper considers a model suggested by Bensoussan, accompanied by co-authors, \cite{GoBe}, where the decision of finite number of $N$ players of large population of Nash-game-players are approximated by $N$ representative agents. So the Bellman system is a system of parabolic equations coupled with a forward backward mean field equation.

In principle, the dependence of the coefficients of the data in the nonlinearities of the equation  may be a functional one. But in order to have a first insight in the difficulties and in order to simplify the presentation, we confine to a \emph{point-wise} dependence of Hamiltonians with respect to the mean field variable here. Generally, the obtaining the existence of solution of the Bellman mean field dependent system with growth with respect to the mean field variable in nonlinearities is the critical subject. Without additional assumptions only the  poor growth behaviour is permitted. For the case of scalar Bellman equations, for obtaining the global solvability, \cite{Por13,Po2} gives a quite exhaustive analysis of the growth conditions for the Hamiltonians concerning the dependence of $\nabla u$ ($u$ is a generalized value function) and the mean field $m$, which appears in the pay off functional.

In comparison, in the present paper, we restrict ourselves to Hamiltonians which grow quadratically with respect to $\nabla u$, which is from the point of view of PDE analysis the most interesting case. A related paper \cite{BeBrFr16}, where a mean field dependence of the pay off functional is assumed, but no such a dependence of $m$ in the dynamics of the system is considered. Also in \cite{BeBrFr16}, the growth properties of the data with respect to $m$, were crucial to obtain global solvability of the problem.

In this paper we goes much beyond the scalar theory and the theory developed in \cite{BeBrFr16} and obtain the existence result for much larger class of problems. As a key tool for the existence of a solution we use the method of sub and super solution used in the context of Bellman systems in \cite{BeBuFr12}.

\section{Derivation of the system}
\label{S1}

In this paper we study a system of partial differential equations which arises as a necessary condition of a Nash--Point--problem for Vlasov--McKean--functionals
\begin{equation}
\label{1.1}
\mathcal{J}^i(\bv)= \int_0^T \int_{\Omega} m(\bv) f^i(\cdot, \bv, m(\bv)) \dx \dt + \int_{\Omega} u_T^i m(T,\cdot) \dx,
\end{equation}
where $(0,T)$ is a given time interval, $\Omega\subset \mathbb{R}^d$ is a cube $(0,1)^d$ and $Q:=(0,T)\times \Omega$ is a space-time cylinder. The function $\bv:=(v^1,\ldots, v^N)$ with $N\in \mathbb{N}$ is the vector of the control functions, i.e.,
$$
v^i:Q \to \mathbb{R}^M
$$
is the control of the $i$-th player, where $i=1, \ldots, N$ and $M\in \mathbb{N}$ is given. For every $i=1,\ldots, N$ the function
$$
f^i:Q \times \mathbb{R}^{MN} \times \mathbb{R} \to \mathbb{R}
$$
is the so-called pay off function of the $i$-th player. Finally, the function
$$
m:  Q \to \mathbb{R},
$$
is the so-called mean field variable. This means that for a given $\bv$, it is a weak nonnegative solution of the following parabolic equation (``mean field equation")
\begin{equation}
\label{1.2}
\partial_t m - \Delta m + \diver \left(m \bg(\cdot, \bv, m)\right) =0
\end{equation}
that is supposed to be satisfied in the space-time cylinder  $Q$,  is completed be spatially periodic boundary conditions (with respect to the unit cube $\Omega$) and by the initial data
\begin{equation}
\label{1.3} m(0)=m_0\ge 0 \qquad \textrm{ in }\Omega.
\end{equation}
Here,
$$
\bg: Q \times \mathbb{R}^{NM} \times \mathbb{R} \to \mathbb{R}^d
$$
is a given mapping and we postpone the discussion about its structure to the end of the section.

For a bounded $\bv$, under natural assumptions on the mapping $\bg$, the problem \eqref{1.2}--\eqref{1.3} has a unique solution
\begin{equation}\label{start}
m\in L^{\infty}(0,T; L^{\infty}(\Omega))\cap L^2(0,T; W^{1,2}_{per}(\Omega))\cap W^{1,2}(0,T; (W^{1,2}_{per}(\Omega))^*),
\end{equation}
which allows us to define \eqref{1.1}. Since the mean field variable depends on the choice of $\bv$, we will frequently also write $m=m(\bv)$ whenever the couple $(m,\bv)$  solves \eqref{1.2}, which should not be understood as an algebraic relation.

In \eqref{start} we use the standard notation for Bochner, Sobolev and Lebesgue spaces and the subscript ``per" indicating the periodicity with respect to $\Omega$ and this notation will be kept through the whole paper. In what follows, we also omit writing the dependence of function on $(t,x)$ explicitly to shorten all formulae, i.e., we use the following abbreviation $\bef(m,\bv)$ for $\bef(t,x,m(t,x),\bv(t,x))$ or for $\bef(\cdot,m,\bv)$ in what follows, where $\bef:=(f^1, \ldots, f^N)$.

Having the mean field variable $m$ and the control function $\bv$, we can define a ``pre-version" of the so-called Bellman system for a further function $\bu=(u^1,\ldots, u^N):Q \to \mathbb{R}^N$ via the backward parabolic system
\begin{equation}\label{1.6}
\begin{split}
-\partial_t{\bu} - \Delta \bu &= \bef(\bv,m) + m\bef_m(\bv,m)\\
&\qquad  \nabla \bu \left[ \bg(\bv,m) + m\bg_{m} (\bv,m)\right]=:L(m,\bv,\nabla \bu),
\end{split}
\end{equation}
which is supposed to be satisfied in $Q$, equipped with the $\Omega$-periodic boundary conditions and completed by the initial condition
$$
\bu(T)=\bu_T \qquad \textrm{ in }\Omega.
$$
We call this system the ``pre-Bellman equation" since $\bv$ has not been replaced by a feed back formula
\begin{equation}\label{1.7}
\bv(t,x):=\bw (t,x,\nabla \bu(t,x),m(t,x))
\end{equation}
yet and we introduce the meaning of \eqref{1.7} in the next subsection. Moreover, $L=(L^1,\ldots, L^N)$ are the so-called modified Lagrangians\footnote{We use here the word modified since they differs from standard Lagrangians in Bellman systems, which is however here caused by the fact that $\bef$ and $\bg$ depend on $m$.}.

It is evident, that \eqref{1.2} and \eqref{1.6} does not form a closed problem and one needs to connect $\bv$ with $m$ and $\nabla \bu$ via some relationship. It will be shown in the next subsection that the necessary condition for classical Nash--Point problem may serve as such constraint. In addition assuming further certain qualitative properties of $\bef$, we will be able to give a good meaning to the feed back formula \eqref{1.7} and thus to  avoid the presence of the control variable $\bv$ in the analysis. The main goal of the paper is to introduce certain structural assumptions on $\bef$ and $\bg$ such that they describe very general mean field dependent Bellman system on one hand, and for which we can establish the existence of a weak solution on the other hand.

\subsection{Derivation of the full system} \label{SS1.1}
Finally, we need to close the problem \eqref{1.2} and \eqref{1.6} by an algebraic condition which is necessary condition for the Nash point of functionals  $J$'s. The classical Nash--Point problem reads: For given $\bu_T$, find $\bv\in L^{\infty}(0,T; L^{\infty}(\Omega;\mathbb{R}^{MN}))$ such that
\begin{equation}
\label{1.4}
J^i(\bv)\le J^i(v_1,\ldots, v_{i-1},z,v_{i+1},\ldots, v_N)
\end{equation}
for all $z \in L^{\infty}(0,T; L^{\infty}(\Omega;\mathbb{R}^{M}))$ and corresponding $m$'s, the solutions\footnote{We may define this also for other $L^p(L^q)$ spaces once the uniqueness of $m(\bv)$ is guaranteed.} to \eqref{1.2} with $\bv$ replaced by $\bz:=(v_1,\ldots, v_{i-1},z,v_{i+1},\ldots, v_N)$.

The classical version treats the case where $f^i$ and $\bg$ do not depend on $m$. In that case, the problem \eqref{1.4} is purely analytical (meaning stochastic free) formulation of a stochastic differential game driven by the dynamics
$$
\frac{d}{dt}\bx=\bg(t,\bx,\bv).
$$
In recent years, interest came up to study cases with $m$-dependence of the pay-off $f^i$ and/or the dynamics $\bg$. From PDE's point of view, this leads to new interesting version of the Bellman system.

Although, it is not know whether the problem \eqref{1.1} admits a Nash--point, we derive in what follows certain necessary conditions that must be fulfilled by the hypothetical Nash--point, which finally allow us to connect the pre-Bellman system \eqref{1.6} with the mean field equation \eqref{1.2} via the feed back formula \eqref{1.7} or its ``equivalent".

Under natural assumption on the data (cf. Section~\ref{SA}) and, say, $\bv\in L^{\infty}(0,T; L^{\infty}(\Omega;\mathbb{R}^{NM}))$, it is easy to see that the Gateux derivatives of the $J^i$ and of $m(\bv)$ exist. For
$$
\bz^i:=(\underset{i-1}{\underbrace{0,\ldots,0}},z,\underset{N-i}{\underbrace{0,\ldots,0}})
$$
with arbitrary smooth $\Omega$-periodic function $z:Q\to \mathbb{R}^M$, we obtain that
$$
M^i:=\left.\frac{d}{ds} m(\bv + s\bz)\right|_{s=0}
$$
with $i=1,\ldots, N$ satisfies
\begin{equation}\label{Gateu}
\left.
\begin{aligned}
&\partial_t M^i - \Delta M^i= \\
&-\diver \left( M^i\bg(\bv,m)+m\bg_{v^i}(\bv,m)\cdot z + mM^i \bg_m(\bv,m)\right)
\end{aligned}
\right\}  \textrm{ in }Q
\end{equation}
and is completed by the initial condition
$$
M^i(0,x)=0 \quad \textrm{ a.e. in }\Omega.
$$
Here, we use the subscript to abbreviate the notion of partial derivative, i.e., $\bg_{\bv}(\bv,m):= \partial_{\bv}\bg(\bv,m)$ and $\bg_m(\bv,m):= \partial_{m}\bg (\bv,m)$.
Furthermore, assuming that $\bv$ is the Nash--equilibrium, we have for all $i=1,\ldots N$ that
\begin{equation}
\label{1.5}
\begin{split}
0&= \left.\frac{d}{ds} J^{i}(\bv+s\bz^i)\right|_{s=0}\\
&= \int_Q  \left(mf^i_{v^i}(\bv,m)\cdot z + M^if^i(\bv,m)+mM^if^i_{m}(\bv,m) \right)\dx \dt\\
&\qquad  + \int_{\Omega} M^i(T)u^i(T) \dx
\end{split}
\end{equation}
for arbitrary smooth $\Omega$--periodic function $z$. Notice that here $m:=m(\bv)$, i.e., $m$ solves \eqref{1.2} with $\bv$. To evaluate the terms not involving explicitly $z$, we use the equations \eqref{1.6} and \eqref{Gateu}. First, multiplying \eqref{Gateu} by $u^i$, integrating over $Q$ and using integration by parts (note that $M^i(0)=0$), we deduce
\begin{equation}\label{Gateu2}
\begin{aligned}
&\int_{\Omega} M^i(T)u^i(T)\dx +\int_{Q} \nabla  M^i \cdot \nabla u^i \dx \dt-\int_{Q}M^i \partial_t u^i \dx \dt\\
&= \int_{Q} ( M^i\bg(\bv,m)+m\bg_{v^i}(\bv,m)\cdot z + mM \bg_m(\bv,m))\cdot \nabla u^i \dx \dt.
\end{aligned}
\end{equation}
Next, multiplying the $i$-th equation in \eqref{1.6} by $M^i$, we observe
\begin{equation}\label{1.68}
\begin{split}
&-\int_{Q}\partial_t u^i M^i\dx\dt +\int_Q \nabla u^i \cdot \nabla M^i\dx \dt  \\
&= \int_QM^i(f^i(\bv,m) + mf^i_m(\bv,m))\dx \dt\\
&\qquad + \int_Q M^i (\bg(\bv,m) + m\bg_{m} (\bv,m))\cdot \nabla u^i\dx \dt.
\end{split}
\end{equation}
Finally, subtracting \eqref{1.68} from \eqref{Gateu2}, we obtain the following identity
\begin{equation}\label{Gateu3}
\begin{aligned}
\int_{\Omega} M^i(T)u^i(T)\dx &= -\int_Q M^i(f^i(\bv,m)+mf^i_m(\bv,m))\dx \dt \\
&\quad +\int_Q m\bg_{v^i}(\bv,m)\cdot z\dx \dt.
\end{aligned}
\end{equation}
Thus, using this relation in the necessary condition \eqref{1.5}, we see that
$$
\int_Q f^i_{v^i}(\bv,m)\cdot z + (\nabla u^i \otimes z) \cdot \bg_{v^i} (\bv,m) \dx \dt =0
$$
for all $i=1,\ldots, N$ and all smooth $\Omega$-periodic $z$. This consequently leads to the necessary compatibility condition
\begin{equation}
\label{1.8}
f^i_{v^i}(\bv,m)  + \nabla u^i \cdot \bg_{v^i} (\bv,m) =0  \qquad \textrm{ in } Q.
\end{equation}

Thus, now we have a closed system od equations. Namely, \eqref{1.2}, \eqref{1.6} and \eqref{1.8} forms a well-defined problem for which we want to establish our analytical result. Indeed, the first one will deal just with \eqref{1.2}, \eqref{1.6} and \eqref{1.8} and lead to the uniform a~priori estimate for $(m,\bv,\bu)$. However, to obtain also the existence result, we shall require that for a given $(m,\nabla \bu)$, we can find a unique $\bv$ solving \eqref{1.8}. For such a solution we define the feed back formula \eqref{1.7} as
$$
\bw(m,\nabla \bu):=\bv
$$
and replacing $\bv$ in \eqref{1.6} by the feed back formula, we obtain
\begin{equation}
\label{1.12}
-\partial_t \bu  - \Delta \bu = H(\nabla \bu, m):=L(m,\bw(m,\nabla \bu),\nabla \bu),
\end{equation}
Similarly, we replace $\bv$ in the mean field equation \eqref{1.2} and obtain the backward forward system
\begin{equation}\label{ff}
\begin{split}
\partial_t m - \Delta m &=-\diver \left(m \bg(\bw(m,\nabla \bu), m)\right) ,\\
-\partial_t \bu  - \Delta \bu &= H(\nabla \bu, m),
\end{split}
\end{equation}
which does not contain a control function $\bv$. Nevertheless, this system is equivalent to \eqref{1.2} and \eqref{1.6} provided that $\bv$ is defined such that it satisfies \eqref{1.8}. The second result of the paper will be therefore established for \eqref{ff}, provided that the feed back formula $\bw$ is well defined.

\subsection{Structural assumptions on $\bef$ and $\bg$}\label{SA}
We keep the notation from the introduction here. Through the whole text, we assume that all partial derivatives $\bef_{\bv}$, $\bg_{\bv}$, $\bef_m$ and $\bg_m$ with respect to $\bv$ and $m$, respectively, exist and together with $\bef$ and $\bg$ are Carath\'{e}odory mappings, i.e., are measurable with respect to $(t,x)$ for all $(\bv,m)$ and for almost all $(t,x)$ they are continuous with respect to $(\bv,m)$. This assumption will not be mentioned explicitly in what follows but we rather assume it implicitly in all statements. In what follows we also assume that $K$ and $C_i$ are positive constants and the same for constants $r,s\ge 0$ which will be used to denote certain powers. Furthermore, we will frequently use $C$ to denote a generic constant that may change from line to line but will depend only on data. In case, it will depend on some important quantity, it will be clearly denoted in the text.

First, we state the assumptions for $\bef:Q\times \mathbb{R}^{NM} \times \mathbb{R} \to \mathbb{R}^N$. We assume that $\bef$ and $\bef_{\bv}$ satisfies the following growth condition
\begin{equation}\label{2.2}
\begin{split}
|\bef(m,\bv)| + m|\bef_{m}(m,\bv)|&\le K(m^r+1)|\bv|^2 + Km^{2s_0},\\
|\bef_{\bv}(m,\bv)| &\le K(m^r+1)|\bv| + Km^{s_0}.
\end{split}
\end{equation}
for all $\bv$ and all $m\ge 0$. In addition, we assume the following one sided estimate for $f^i$: There exists $\alpha\in[0,1)$ such that  for all $i=1,\ldots, N$ there holds
\begin{equation}\label{2.21}
f^i(m,\bv)+mf^i_{m}(m,\bv)\le K\left(1+f^i(m,\bv)+|\bv|^{\alpha+1}+m^{2s_0}\right)
\end{equation}
and also one sided sum coerciveness, i.e., we assume
\begin{equation}\label{2.3}
\sum_{i=1}^N f^i(m,\bv)+mf^i_{m}(m,\bv)\ge C_0(m^r+1)|\bv|^2 - K(m^{2s_0}+1).
\end{equation}
Concerning the structural assumptions on $\bef$ we just assume that
\begin{equation}\label{convex}
\textrm{For all $i$ the function $f^i$ is convex with respect to $v^i$.}
\end{equation}
We also prescribe the behaviour of $\bef$ at $0$, i.e., for all $i=1,\ldots, N$ we assume that
\begin{equation}\label{at0}
f^i(m,\bv)|_{v^i=0}\le K(1+m^{2s_0}+|\bv|^{\alpha+1})
\end{equation}
and finally the coerciveness of $f^i_{v^i}$, i.e., we assume that
\begin{equation}\label{fv}
C_0(m^r+1)|v^i|^2 \le f^i_{v^i}(m,\bv) \cdot v^i + K(1+m^{2s_0}+|\bv|^{\alpha+1}).
\end{equation}

Next, we focus on the assumptions on $\bg$. The standard ones are related to the growth estimates, which will be supposed to be given by
\begin{equation}
\begin{split}
m|\bg_{m}(m,\bv)| + |\bg(m,\bv)|&\le K((m^s+1)|\bv|+m^{s_0}+1),\\
|\bg_{\bv}(m,\bv)|&\le K(m^s+1).
\end{split}\label{2.6}
\end{equation}

The forthcoming structural assumptions on $\bg$ and also on $\bef_{\bv}$ are in fact the key restrictions of the paper.
First, to simplify the further analysis, we shall assume that\footnote{The linearity of $\bg$ with respect to $\bv$ is in fact not a necessary assumption. A more delicate here is just behaviour of $\bg$ with respect to $m$.}
\begin{equation}\label{g-def}
\bg = \sum_{j=1}^N b_1(m)A^j(\cdot)v^j + \bb_0(m),
\end{equation}
where $\bb_0$ does not depend on $(t,x)$. The matrices $A^j:Q \to \mathbb{R}^{d\times M}$ are given functions of $(t,x)$, i.e., $(A^j)_{ik}:=A^j_{ik}$ with $i=1,\ldots, d$ and $k=1,\ldots, M$ and the meaning of $A^j v^j$ is
$$
(A^j v^j)_i:=\sum_{k=1}^M A^{j}_{ik} v^j_k.
$$
Hence, as it is  assumed we have a $d$-dimensional vector function $\bg=(g_1,\ldots, g_d)$. The inhomogeneities  $\bb_0:\mathbb{R}\to \mathbb{R}^d$ and $b_1:\mathbb{R}\to \mathbb{R}$ are given. Concerning the assumptions on functions $b_1$ and $\bb_2$, we require that
\begin{equation}
|b_1(m)|\le K(m^s+1), \qquad |\bb_0(m)|\le K(m^{s_0}+1) \label{as-g1}
\end{equation}
and  for matrices $A^j$ we assume that
\begin{equation}
|A^j(\cdot )|\le K. \label{as-A1}
\end{equation}
Furthermore, one of the essentially required properties of $A^j$ is that they have the same range, i.e., we assume that for all $i,j=1,\ldots, N$, almost all $(t,x)$ and all $\bz \in \mathbb{R}^{Nd}$ there holds
\begin{equation}
|\bz A^j(t,x)|\le C_1|\bz A^i(t,x) |.\label{as-A2}
\end{equation}
For derivatives of $\bb_0$, we need that
\begin{equation}
\begin{split}
|m\partial_m\bb_0(m)|&\le Km(m+1)^{s_0-1},\\
|m^2\partial_{mm}\bb_0(m)|&\le Km^2(m+1)^{s_0-2}. \label{as-g2}
\end{split}
\end{equation}
Finally, for the derivative of $b_1$ with respect to $m$,  we introduce a certain ``smallness assumption": There exists $\delta\in[0,1)$ such that for all $b\in \mathbb{R}_+$ and all $\bv\in \mathbb{R}^{NM}$ there holds
\begin{equation}\label{b1-as}
\begin{split}
&C_1\sqrt{N}\frac{|m\partial_m b_1(m)|}{|b_1(m)|}|\bv| \sum_{i=1}^M |f^i_{\bv^i}|\\
&\qquad \le \sum_{i=1}^N(f^i(m,\bv)+m\partial_m f^i(m,\bv)) + K(1+m^{2s_0}).
\end{split}
\end{equation}
Although the above assumption seems to be complicated, it naturally appears in the first a~priori estimate of the problem. In the next subsection, we shall show that in many cases, the assumption \eqref{b1-as} still allows very general behaviour of all quantities with respect to $m$, namely, the case when $\bf$ can dominate the function $\bg$ in certain sense. In addition, since \eqref{b1-as} may seem to be complicated, we can replace it by
\begin{equation}\label{as-gamma}
|m\partial_m b_1(m)|\le \gamma |b_1(m)| \textrm{ with } \gamma \le \frac{C_0}{2(C_1^2+N^2)}
\end{equation}
since then \eqref{b1-as} follows easily from the previous assumptions.

Please, observe here that if $\bg$ is given by \eqref{g-def}, then \eqref{1.8} reduces to
\begin{equation}
\label{1.81}
f^i_{v^i_{j}}(\bv,m)  + b_1(m) \sum_{k=1}^d \partial_{x_k} u^i A^i_{kj} =0  \qquad \textrm{ in } Q
\end{equation}
which must be valid for all $i=1,\ldots, N$ and all $j=1,\ldots, M$.


The above assumptions are sufficient for establishing formal a~priori estimates. However, for getting also the existence of a weak solution, we need to give a well meaning to the feed back formula, or in other words, we need to guarantee the unique solvability of \eqref{1.81} for given $m$ and $\nabla u$. There can be introduced many assumptions that wold lead to such a goal but we follow here the most standard one, which is the monotone operator approach. For a given $m$, we define the mapping $T:\mathbb{R}^{NM} \to \mathbb{R}^{NM}$ defined by
$$
T(\bv)=\left(\frac{\partial}{\partial v^1} f^1(\bv,m), \ldots,\frac{\partial}{\partial v^N} f^N(\bv,m)\right)
$$
and assume that it is continuous with respect to $(m,\bv)$ and measurable with respect to $(t,x)$ and the strictly monotone,  i.e., for all $\bv \neq \tilde{\bv}$
\begin{equation}
\label{1.11}
(T(\bv)-T(\tilde{\bv}) \cdot (\bv - \tilde{\bv}) >0.
\end{equation}
Then using also the assumption \eqref{fv}, we see that it satisfies
\begin{equation}
\label{1.10}
\frac{T(\nu)\cdot \nu}{|\nu|} \to \infty \qquad \textrm{ as } |\nu| \to \infty
\end{equation}
and from the standard monotone operator theory, we can conclude the existence of a unique $\bv$ solving \eqref{1.81} and also consequently, we obtain that the feed back law \eqref{1.7} is well defined.

\subsection{Prototypical example}
Our prototype example is the following. For $f$, we assume a structure
\begin{equation}
f^i(m,\bv):=(m+1)^r|v^i|^2 + B^i\cdot \bv + K(1+m^{2s_0}),\label{f-pro}
\end{equation}
where $B^i$'s are arbitrary bounded measurable matrices. Next, for $b_1$ we consider
\begin{equation}
b_1(m):=(m+1)^s
\end{equation}
and for $\bb_2$ and matrices $A^i$, we just require that \eqref{as-A1}, \eqref{as-A2} and \eqref{as-g2}. With such a choice, all assumptions \eqref{2.2}--\eqref{as-g2} and also \eqref{1.11} are satisfied ad we shall just to show what is the meaning of \eqref{b1-as}. A very direct computation leads to the necessary condition
$$
sm\le \frac{\gamma}{2N} \left(1+(r+1)m\right),
$$
which is surely satisfied whenever
\begin{equation}
s<\frac{r}{2N}.\label{prot-asss}
\end{equation}

\subsection{Statement of main results}
The main result of the paper is twofold. First, we give the uniform a~priori estimate result which holds for all sufficiently smooth solutions provided that parameters satisfy the assumptions stated above.
\begin{Theorem}\label{T1}
Let $\bg$ and $\bef$ satisfy \eqref{2.2}--\eqref{b1-as}. Then any sufficiently regular solution $(m,\bv,\bu)$ to \eqref{1.2}, \eqref{1.6} and \eqref{1.8} satisfies the following estimate
\begin{equation}\label{T-est}
\begin{split}
&\sup_{t\in (0,T)} \left(\|m(t)\|_{\sigma} + \|\bu(t)\|_{\infty}\right)+\int_{Q}|\nabla \bu|^2 + (m+1)^{\sigma -2}|\nabla m|^2 \dx \dt\\
&+\int_Q m^{2s_0+1}+ (m+1)(m^{r}+1)|\bv|^2 \dx \dt \le C(\|\bu_T\|_{\infty}, \|m_0\|_{\sigma}).
\end{split}
\end{equation}
where
\begin{equation}
\sigma:= r-2s+1 \label{ddsig}
\end{equation}
provided that
\begin{equation}\label{ffres}
0\le \min \left\{\frac{4(2s_0(d+2)-1)_+}{\sigma(d+2)-d-(2s_0-\sigma +1)_+(d+2)} , \frac{2s_0 (d+2)}{\sigma(d+2)-d} \right\}<1.
\end{equation}
and
\begin{equation}\label{crcr}
r\ge2s.
\end{equation}
\end{Theorem}
Next, we state the second main theorem of the paper, which is the existence result.
\begin{Theorem}\label{T2}
Let $\bg$ and $\bef$ satisfy \eqref{2.2}--\eqref{b1-as} and \eqref{1.11}. Assume that \eqref{ffres}--\eqref{crcr} is fulfilled. Then for arbitrary $\bu_T\in L^{\infty}$ and nonnegative $m_0\in L^{\sigma}(\Omega)$ with $\sigma$ fulfilling \eqref{ddsig} there exists a weak solution satisfying the estimate \eqref{T-est} and fulfilling or almost all $t\in (0,T)$
\begin{equation}\label{wf-m}
\langle\partial_t m, \varphi\rangle  +\int_{\Omega} \nabla m \cdot \nabla \varphi - (m \bg( \bv, m)\cdot \nabla \varphi =0
\end{equation}
\begin{equation}\label{wf-u}
\begin{split}
-\langle \partial_t \bu, \bz \rangle  +\int_{\Omega} \nabla \bu \cdot \nabla \bz \dx &= \int_{\Omega} (\bef(\bv,m) + m\bef_m(\bv,m))\cdot \bz\dx \\
& \quad +\int_{\Omega} \nabla \bu \left[ \bg(\bv,m) + m\bg_{m} (\bv,m)\right]\cdot  \bz \dx
\end{split}
\end{equation}
and completed by the relationship between $\bv$ and $(m,\nabla \bu)$
\begin{equation}
\label{wf-v}
f^i_{v^i_{j}}(\bv,m)  + b_1(m) \sum_{k=1}^d \partial_{x_k} u^i A^i_{kj} =0  \qquad \textrm{ in } Q.
\end{equation}
\end{Theorem}
To lustrate the power of the result, we just consider our prototypical example. First in case that $s_0=0$, we see that the only restriction is
$$
2Ns<r.
$$
In the opposite extreme case, i.e., if $r=s=0$, then
$$
s_0 <\frac{1}{2(d+2)}.
$$

\section{Algebraic estimates for Lagrangians and Hamiltonians}\label{Sec3}
In this section, we derive basic algebraic inequalities that are satisfied for Lagrangians and consequently also for Hamiltonians provided that $\bef$ and $\bg$ satisfy the assumption introduced in Section~\ref{SA}, namely the assumptions \eqref{2.2}--\eqref{b1-as}. The key observation is that under these assumptions, the Lagrangians  satisfy the the lower sum corecivness and the proper upper estimates. It will be also evident from the estimates below why we require $2s\le r$ in main results of the paper.
\begin{Lemma}\label{L1}
Let $\bef$ and $\bg$ satisfy \eqref{2.2}--\eqref{b1-as}. Then there exists a constant $C>0$ and an $\varepsilon_0\in (0,1/(2N))$ such that for almost all $(t,x)$, all $(m,\bv,\nabla \bu)$ fulfilling \eqref{1.8} and all all $i=1,\ldots, N$ there hold:
\begin{itemize}
\item[i)] sum coerciveness
\begin{equation}\label{Le1}
\begin{split}
\sum_{i=1}^N L^i(m,\bv,\nabla \bu) &\ge\frac{C_0}{2}(m^r+1)|\bv|^2 - C(m^{2s_0}+1)\\
&-C\left|\nabla \sum_{i=1}^N  u^i \right|^2 \left(1+\frac{m^{2s}+1}{m^r+1}\right);
\end{split}
\end{equation}
\item[ii)] upper bound
\begin{equation}\label{Le2}
\begin{split}
&L^i(m,\bv,\nabla \bu) - \varepsilon_0\sum_{j=1}^N L^j(m,\bv,\nabla \bu) \\
&\quad \le C(1+m^{2s_0}) + C \left|\nabla\left(u^i- \varepsilon_0 \sum_{i=j}^N u^j\right)\right|^2\left(1+\frac{m^{2s}+1}{m^r+1}\right);
\end{split}
\end{equation}
\item[iii)] global bound
\begin{equation}\label{Le3}
\begin{split}
&|L^i(m,\bv,\nabla \bu)|\\
&\quad \le C \left(1+m^{2s_0}+(1+m^r)|\bv|^2 + |\nabla \bu|^2\left(1+\frac{m^{2s}+1}{m^r+1}\right) \right).
\end{split}
\end{equation}
\end{itemize}
\end{Lemma}
\begin{proof}
We start with the proof of \eqref{Le1}. Using the definition of $L$ in \eqref{1.6} we get the identity
\begin{align*}
\sum_{i=1}^N L^i(m,\bv,\nabla \bu) &= \sum_{i=1}^N \left(f^i(\bv,m) + mf^i_m(\bv,m)\right)\\
&\quad+\sum_{i=1}^N \nabla u^i \cdot ( \bg(\bv,m) + m\bg_{m} (\bv,m)).
\end{align*}
Hence, using \eqref{2.3} for the first part and \eqref{2.6} for the second part, we observe that
\begin{align*}
\sum_{i=1}^N L^i(m,\bv,\nabla \bu) &\ge  C_0(m^r+1)|\bv|^2 - K(m^{2s_0}+1)\\
&-K\left|\nabla \sum_{i=1}^N  u^i \right|   ((m^s+1)|\bv|+m^{s_0}+1)\\
&\ge  \frac{C_0}{2}(m^r+1)|\bv|^2 - C(m^{2s_0}+1)\\
&-C\left|\nabla \sum_{i=1}^N  u^i \right|^2 \left(1+ \frac{m^{2s}+1}{m^r+1}\right),
\end{align*}
where for the second estimate we used the Young inequality. This finishes the proof of \eqref{Le1}.

Next, we look for \eqref{Le2}. Using the definition of Lagrangians, we directly obtain  for arbitrary $\varepsilon>0$ the following identity
$$
\begin{aligned}
I&:=L^i(m,\bv,\nabla \bu) - \varepsilon \left(\sum_{i=1}^N L^i(m,\bv,\nabla \bu)\right)\\
&=f^i(m,\bv) + mf^i_m(m,\bv) - \varepsilon\left(\sum_{j=1}^N f^j(m,\bv) + mf^j_m(m,\bv)\right)\\
&+ \nabla\left(u^i- \varepsilon \sum_{j=1}^N u^j\right) \cdot\left( \bg(m,\bv) + m\bg_{m} (m,\bv)\right).
\end{aligned}
$$
Next, using \eqref{2.21}, \eqref{2.3}, \eqref{2.6} and the Young inequality, we have
\begin{equation}\label{P1}
\begin{aligned}
I&\le C(1+m^{2s_0} +|\bv|^{\alpha+1}) + Kf^i(m,\bv)- \varepsilon C_0(m^r+1)|\bv|^2 \\
&\quad + K\left|\nabla\left(u^i- \varepsilon \sum_{i=1}^N u^i\right)\right|\left(1+m^{s_0}+(m^s+1)|\bv|\right)\\
&\le C(\varepsilon)(1+m^{2s_0} ) + Kf^i(m,\bv)- \frac{\varepsilon C_0}{2}(m^r+1)|\bv|^2 \\
&\quad + C(\varepsilon) \left(1+\frac{m^{2s}+1}{m^r+1}\right)\left|\nabla\left(u^i- \varepsilon \sum_{i=1}^N u^i\right)\right|^2.
\end{aligned}
\end{equation}
It remains to estimate the term with $f^i$. We use the convexity of $f^i$ with respect to $v^i$ and the assumption \eqref{at0}. Then to evaluate $f^i_{v^i}$ we also use the constraint \eqref{1.8}, which however in our case reduces to \eqref{1.81} due to the structural assumptions \eqref{g-def}. Doing so, we get
\begin{equation*}
\begin{aligned}
&Kf^i(m,\bv)\le Kf^i(m,\bv)|_{v^i=0} + Kf^i_{v^i}(m,\bv) \cdot v^i \\
&\le C(1+m^{2s_0}+|\bv|^{\alpha +1}) -Cb_1(m)\sum_{k=1}^d\sum_{j=1}^N \partial_{x_k}u^i A^{i}_{kj}v^i_j \\
&= C(1+m^{2s_0}+|\bv|^{\alpha +1}) -Cb_1(m)\sum_{k=1}^d\sum_{j=1}^N \partial_{x_k}\left(u^i-\varepsilon\sum_{\ell=1}^N u^{\ell}\right) A^{i}_{kj}v^i_j \\
&\quad -\varepsilon Cb_1(m)\sum_{k=1}^d\sum_{j=1}^N \sum_{\ell=1}^N \partial_{x_k}u^{\ell} A^{i}_{kj}v^i_j.
\end{aligned}
\end{equation*}
Then we estimate terms involving $A^{\ell} \nabla u^i$ with the help of \eqref{as-g1}, \eqref{as-A2} and \eqref{1.81}  as follows
\begin{equation}
\begin{aligned}
&Kf^i(m,\bv)\\
&\le C(1+m^{2s_0}+|\bv|^{\alpha +1}) +C(m^s+1)\left|\nabla (u^i-\varepsilon\sum_{\ell=1}^N u^{\ell})\right||\bv| \\
&\quad +\varepsilon C|v^i|\sum_{\ell=1}^N |\nabla u^{\ell} A^{\ell}||b_1(m)|\\
&\le C(1+m^{2s_0}+|\bv|^{\alpha +1}) +C(m^s+1)\left|\nabla (u^i-\varepsilon\sum_{\ell=1}^N u^{\ell})\right||\bv| \\
&\quad +\varepsilon C|v^i||\bef_{\bv}(m,\bv)|.
\end{aligned}\label{P3}
\end{equation}
Finally, we focus on estimate of $v^i$. It follows from \eqref{fv}, \eqref{as-A2}, \eqref{as-g1} and \eqref{1.81} that
\begin{equation*}
\begin{split}
C_0(m^r+1)|v^i|^2&\le f^i_{v^i}\cdot v^i + K(1+m^{2s_0}+|\bv|^{1+\alpha})\\
&=-b_1(m)\nabla u^i A^i \cdot v^i + K(1+m^{2s_0}+|\bv|^{1+\alpha})\\
&=-b_1(m)\nabla \left(u^i-\varepsilon \sum_{\ell=1}^N u^{\ell}\right) A^i \cdot v^i\\
 &\quad -\varepsilon b_1(m)\sum_{\ell=1}^N \nabla u^{\ell} A^i \cdot v^i+ K(1+m^{2s_0}+|\bv|^{1+\alpha})\\
&\le C(m^s+1)\left|\nabla \left(u^i-\varepsilon \sum_{\ell=1}^N u^{\ell}\right)\right||v^i|\\
 &\quad +\varepsilon C|v^i| |b_1(m)|\sum_{\ell=1}^N |\nabla u^{\ell} A^{\ell}|+ K(1+m^{2s_0}+|\bv|^{1+\alpha})\\
 &\le C\frac{m^{2s}+1}{m^r+1}\left|\nabla \left(u^i-\varepsilon \sum_{\ell=1}^N u^{\ell}\right)\right|^2+ C_0\frac{m^r+1}{2}|v^i|^2\\
 &\quad +\varepsilon C|\bv||\bef_{\bv}(m,\bv)|+ K(1+m^{2s_0}+|\bv|^{1+\alpha}),
\end{split}
\end{equation*}
where for the last inequality, we used the Young inequality and the structural constraint \eqref{1.81}.
Hence absorbing the corresponding term to the left hand side, using \eqref{2.2} and the Young inequality, we get
\begin{equation}\label{P4}
\begin{split}
(m^r+1)|v^i|^2&\le C\frac{(m^{2s}+1)}{m^r+1}\left|\nabla \left(u^i-\varepsilon \sum_{\ell=1}^N u^{\ell}\right)\right|^2\\
 &+\varepsilon C(m^r+1)|\bv|^2 + C(1+m^{2s_0}+|\bv|^{1+\alpha}).
\end{split}
\end{equation}
Next, using the Young inequality in \eqref{P3}, we find
\begin{equation}
\begin{aligned}
&Kf^i(m,\bv)\\
&\le C(\varepsilon)(1+m^{2s_0}) +C\frac{m^{2s}+1}{\varepsilon^2 (m^r+1)}\left|\nabla (u^i-\varepsilon\sum_{\ell=1}^N u^{\ell})\right|^2 \\
&\quad +C\varepsilon^2(m^r+1) |\bv|^2 +\varepsilon C|v^i||\bef_{\bv}(m,\bv)|\\
&\le C(\varepsilon)(1+m^{2s_0}) +C\frac{m^{2s}+1}{\varepsilon^2 (m^r+1)}\left|\nabla \left(u^i-\varepsilon\sum_{\ell=1}^N u^{\ell}\right)\right|^2 \\
&\quad +C\varepsilon^2(m^r+1) |\bv|^2 +\varepsilon^{\frac12} (m^r+1)|v^i|^2 +C \varepsilon^{\frac32}\frac{|\bef_{\bv}(m,\bv)|^2}{m^r+1}.
\end{aligned}\label{P5}
\end{equation}
Finally, we substitute  \eqref{P4} into \eqref{P5} to estimate term with $v^i$ and with the help of the assumption \eqref{2.2} we obtain
\begin{equation}
\begin{aligned}
&Kf^i(m,\bv)\\
&\le C(\varepsilon)\left(1+m^{2s_0}+\frac{m^{2s}+1}{m^r+1}\left|\nabla \left(u^i-\varepsilon\sum_{\ell=1}^N u^{\ell}\right)\right|^2\right) \\
&\quad +C\varepsilon^{\frac32}(m^r+1) |\bv|^2.
\end{aligned}\label{P6}
\end{equation}
Consequently, \eqref{P6} combined with \eqref{P1} directly implies
\begin{equation}\label{P7}
\begin{aligned}
I&\le C(\varepsilon)(1+m^{2s_0})- \frac{C_0}{2}(\varepsilon-C\varepsilon^{\frac32})(m^r+1)|\bv|^2 \\
&\quad + C(\varepsilon) \frac{m^{2s}+1}{m^r+1}\left|\nabla\left(u^i- \varepsilon \sum_{i=1}^N u^i\right)\right|^2.
\end{aligned}
\end{equation}
Thus, if we choose $\varepsilon=:\varepsilon_0 \in (0,2/N)$ such that
$$
\varepsilon_0-C\varepsilon_0^{\frac32}\ge 0,
$$
which is always possible and the maximal value of such $\varepsilon_0$ depends only od the generic constant $C$, we obtain the desired estimate \eqref{Le2}.

The estimate \eqref{Le3} is then a simple combination of \eqref{Le1} and \eqref{Le2}. The proof is finished.
\end{proof}

\section{Uniform a~priori estimates - Proof of Theorem~\ref{T1}}
This section is devoted to the estimates for solution $(m,\bu, \bv)$ of \eqref{1.2}, \eqref{1.6} and \eqref{1.8} that depend only on data of the problem provided that the assumptions \eqref{2.2}--\eqref{b1-as} are satisfied. Here, we proceed rather formally, considering that the solution is sufficiently regular. The justification of such a procedure then will be provided in the proof of the existence result, i.e., in the proof of Theorem~\ref{T2}.

\subsection{Estimates for $m$}
We start with estimates for $m$. First, if $m_0\ge 0$ almost everywhere in $\Omega$, then the standard minimum principle for parabolic equations implies that (sufficiently smooth) solution satisfies $m\ge 0$ almost everywhere in $Q$ as well. Next, setting $\varphi:=1$ in \eqref{wf-m}, we get with the help of nonnegativity of $m$ that
$$
\frac{d}{\dt} \|m(t)\|_1 = 0.
$$
Consequently, we have
\begin{equation}\label{3.1}
\sup_{t\in (0,T)}\|m(t)\|_1 \le \|m_0\|_1 \le C.
\end{equation}

Next, in order to obtain estimates on $(\bu, \bv)$ we need to improve the information about $m$ since the Lagrangians (or Hamiltonians) depend heavily on $m$. Our goal is to prove the starting point inequality
\begin{equation}\label{3.35}
\begin{split}
E:=&\sup_{t\in (0,T)}\|m(t)+1\|_{\sigma}^{\sigma} + \int_Q (m+1)^{\sigma -2}|\nabla m|^2\dx \dt \\
&\le C\int_{Q}(m+1)^{2s_0+1}\dx \dt \\
&\quad + C\|\bu\|_{L^{\infty}(Q)}^2 \left(1+ \int_{Q} (m+1)^{2s_0-\sigma +2}\dx \dt\right),
\end{split}
\end{equation}
where the constant $C$ depends only on data given by assumptions on $\bef$ and $\bg$.

\begin{proof}[Proof of \eqref{3.35}]
We first set $\varphi:=(m+1)^{\sigma-1}$ in \eqref{wf-m}, where $\sigma$ is given by \eqref{ddsig}, i.e.,
\begin{equation*}
\sigma := r+1-2s\ge 1.
\end{equation*}
With such a choice of $\varphi$, we get the identity
$$
\begin{aligned}
\frac{1}{\sigma}& \frac{d}{\dt} \|m(t)+1\|_{\sigma}^{\sigma} + (\sigma-1)\int_{\Omega} (m+1)^{\sigma-2}|\nabla m|^2 \dx \\
&= (\sigma-1) \int_{\Omega}m(m+1)^{\sigma-2}\bg(m,\bv) \cdot \nabla m\dx \\
&=(\sigma-1) \int_{\Omega}m (m+1)^{\sigma-2}b_1(m)\sum_{j=1}^M A^j v^j \cdot \nabla m \dx \\
&\quad  +(\sigma-1)\int_{\Omega} m(m+1)^{\sigma -2}\bb_0(m)\cdot \nabla m \dx,
\end{aligned}
$$
where the second equality follows from the structural assumption on $\bg$, see \eqref{g-def}. Clearly, the second term on the right hand side vanishes due to the integration by parts and spatially periodic boundary conditions. For the first integral we use the Young inequality to conclude
$$
\begin{aligned}
&\frac{1}{\sigma} \frac{d}{\dt} \|m(t)+1\|_{\sigma}^{\sigma} + \frac{(\sigma-1)}{2}\int_{\Omega} (m+1)^{\sigma-2}|\nabla m|^2 \dx \\
&\qquad \le \frac{(\sigma-1)}{2} \int_{\Omega} m^2(m+1)^{\sigma -2} b^2_1(m)\big|\sum_{j=1}^N A^j v^j\big|^2\dx.
\end{aligned}
$$
Thus, using the bounds for $b_1$ and $A^j$, namely \eqref{as-g1} and \eqref{as-A1}, we obtain after integration over $(0,T)$
\begin{equation}\label{3.3}
\begin{split}
\sup_{t\in (0,T)}&\|m(t)+1\|_{\sigma}^{\sigma} + \int_Q (m+1)^{\sigma -2}|\nabla m|^2\dx \dt \\
&\le C(\sigma, \|m_0\|_q,\Omega)\left( 1 + \int_Q m (m+1)^{2s+\sigma-1}|\bv|^2 \dx \dt \right).
\end{split}
\end{equation}

Next,  we set $\varphi:=u^i$ in \eqref{wf-m} and $z:=m$ in the $i$-th equation in \eqref{wf-u}, integrate the result with respect to time and use integration by parts to obtain the following two identities
\begin{equation}\label{A1}
\begin{split}
\int_{\Omega} m(T)u^i(T) - m(0)u^i(0) \dx - \int_Q  m \partial_t u^i \dx \dt   \\
+\int_{Q} \nabla m \cdot \nabla u^i - m \bg(\bv, m) \cdot \nabla u^i \dx \dt  =0
\end{split}
\end{equation}
and
\begin{equation}\label{A2}
\begin{split}
-\int_Q &\partial_t u^i m \dx \dt  +\int_{Q} \nabla u^i \cdot \nabla m \dx dt  \\
&= \int_{Q} m(f^i(\bv,m) + mf^i_m(\bv,m))\dx \dt  \\
&\quad +\int_{Q} m \nabla u^i \cdot \bg(\bv,m) + m^2\bg_{m} (\bv,m)\cdot \nabla u^i \dx \dt.
\end{split}
\end{equation}
Subtracting \eqref{A1} from \eqref{A2} we arrive at identity
\begin{equation}\label{A3}
\begin{split}
\int_{Q} &m(f^i(\bv,m) + mf^i_m(\bv, m)) \dx \dt \\
 &= -\int_{\Omega} m(T)u^i(T) - m(0)u^i(0)\dx  -\int_{Q}  m^2\bg_{m} (\bv,m)\cdot \nabla u^i\\
&\le C\|\bu\|_{L^{\infty}(Q)} -\int_{Q}  m^2\bg_{m} (\bv,m)\cdot \nabla u^i\dx \dt,
\end{split}
\end{equation}
where for the last inequality we used the H\"{o}lder inequality and the uniform bound \eqref{3.1}.
Next, we evaluate the last integral. Using the structural assumption \eqref{g-def} and the growth assumptions \eqref{as-g1}--\eqref{as-g2}, we see that
$$
\begin{aligned}
&-\int_{Q}  m^2\bg_{m} (\bv,m)\cdot \nabla u^i\dx \dt \\
&= -\int_{Q} m^2 \partial_m \bb_0(m) \cdot \nabla u^i \dx \dt  - \sum_{j=1}^N\int_{Q} m^2\partial_m b_1(m)  A^jv^j\cdot\nabla u^i\dx \dt\\
&= \int_{Q}u^i \partial_m( m^2 \partial_m \bb_0(m))\cdot \nabla m \dx \dt \\
&\qquad - \sum_{j=1}^N\int_{Q} m^2\partial_m b_1(m)  A^jv^j\cdot\nabla u^i\dx \dt\\
&\le \int_{Q}K|\bu| m(m+1)^{s_0-1} |\nabla m| +  m^2|\partial_mb_1(m)| |\bv| \left(\sum_{j=1}^N|\nabla u^i A^j|^2\right)^{\frac12}\dx \dt\\
&=:I_1+I_2.
\end{aligned}
$$
To estimate $I_2$ we use \eqref{as-A2} and  the constraint \eqref{1.81} to get
$$
\begin{aligned}
I_2 &\le C_1 \int_{Q}  \frac{m^2 |\partial_m b_1(m)|}{b_1(m)} |\bv|  \left( \sum_{j=1}^N |b_1(m)   \nabla u^i A^i|^2 \right)^{\frac12}\dx \dt \\
&=C_1\sqrt{N}\int_{Q}  \frac{m^2 |\partial_m b_1(m)|}{b_1(m)} \left|f^i_{v^i} \right| |\bv|  \dx \dt.
\end{aligned}
$$
For the term $I_1$ we use the Young and the H\"{o}lder inequalities to obtain for arbitrary $\varepsilon>0$
$$
I_1 \le \varepsilon \int_{Q}(m+1)^{\sigma -2}|\nabla m|^2 \dx \dt + C(\varepsilon)\|\bu\|_{L^{\infty}(Q)}^2 \int_{Q} (m+1)^{2s_0-\sigma +2}\dx \dt.
$$
Finally, substituting these estimates into \eqref{A3}, summing the result over $i=1,\ldots, N$ and using \eqref{b1-as}, we obtain
\begin{equation}\label{Ales}
\begin{split}
\int_{Q} &\sum_{i=1}^N m(f^i(\bv,m) + mf^i_m(\bv, m)) \dx \dt \\
 &\le N\varepsilon \int_{Q}(m+1)^{\sigma -2}|\nabla m|^2 \dx \dt \\
 &+ C(\varepsilon)\|\bu\|_{L^{\infty}(Q)}^2 (1+\int_{Q} (m+1)^{2s_0-\sigma +2}\dx \dt)\\
 &+  C_1\sqrt{N}\int_{Q}  \frac{m^2 |\partial_m b_1(m)|}{b_1(m)} |\bv|\sum_{i=1}^N\left|f^i_{v^i} \right|   \dx \dt\\
 &\le N\varepsilon \int_{Q}(m+1)^{\sigma -2}|\nabla m|^2 \dx \dt \\
 &+ C(\varepsilon)\|\bu\|_{L^{\infty}(Q)}^2 (1+\int_{Q} (m+1)^{2s_0-\sigma +2}\dx \dt)\\
 &+ \int_{Q} \delta\sum_{i=1}^N m(f^i(\bv,m) + mf^i_m(\bv, m)) +K(1+m^{2s_0}) \dx \dt.
\end{split}
\end{equation}
Consequently, since $\delta<1$ we can absorb the last term by the left hand side. Therefore, we can use the sum coerciveness, i.e., the assumption \eqref{2.3}, to deduce that
\begin{equation}\label{A4}
\begin{split}
\int_{Q} &m(m^r+1)|\bv|^2 \dx \dt \\
&\le C\|\bu\|_{L^{\infty}(Q)} +  C(\varepsilon)\|\bu\|_{L^{\infty}(Q)}^2 \int_{Q} (m+1)^{2s_0-\sigma +2}\dx \dt \\
&+ C\int_{Q}(m+1)^{2s_0+1}\dx \dt +C\varepsilon \int_{Q}(m+1)^{\sigma -2}|\nabla m|^2 \dx \dt.
\end{split}
\end{equation}
Finally, using the estimate \eqref{A4} in \eqref{3.3},  we see that (recalling \eqref{ddsig})
\begin{equation*}
\begin{split}
\sup_{t\in (0,T)}&\|m(t)+1\|_{\sigma}^{\sigma} + \int_Q (m+1)^{\sigma -2}|\nabla m|^2\dx \dt \\
&\le C\left( 1 + \int_Q m (m^r+1)|\bv|^2 \dx \dt \right)\\
&\le C\int_{Q}(m+1)^{2s_0+1} +\varepsilon (m+1)^{\sigma -2}|\nabla m|^2 \dx \dt \\
&+ C(\varepsilon)\|\bu\|_{L^{\infty}(Q)}^2 \left(1+ \int_{Q} (m+1)^{2s_0-\sigma +2}\dx \dt\right).
\end{split}
\end{equation*}
Thus, setting $\varepsilon>0$ sufficiently small, we obtain the desired inequality~\eqref{3.35}.
\end{proof}

Our next goal is to improve the information coming from \eqref{3.35} such that the estimate on $E$ depends only o $\bu$ and not on $m$. It means that we want to show that
\begin{equation}\label{3.F2}
\begin{split}
E\le C\left(1+\|\bu\|_{L^{\infty}(Q)}^{\frac{2(\sigma(d+2)-d)}{\sigma(d+2)-d-((2s_0-\sigma +1)_+(d+2))}}\right),
\end{split}
\end{equation}
provided that
\begin{align}
2s_0 + 2s &<r+ \frac{2}{d+2}.\label{restr1}\\
 s_0+2s  &< r+\frac{1}{d+2}.\label{restr2}
\end{align}
The estimate \eqref{3.F2} will be shown by a certain interpolation of \eqref{3.1} and \eqref{3.35}.
\begin{proof}[Proof of \eqref{3.F2}]
First, we  recall the following parabolic interpolation inequality
\begin{equation}
\begin{split}
\int_0^T\|u\|_{\frac{2(d+2)}{d}}^{\frac{2(d+2)}{d}} &\le \int_0^T \|u \|^{\frac{4}{d}}_2 \|u\|_{1,2}^2\dt \\
&\le C\left(\sup_{t\in(0,T)}\|u(t) \|_2^2 + \int_0^T \|\nabla u \|_2^2\dt\right)^{\frac{d+2}{d}}.
\end{split}\label{inter}
\end{equation}
Next, we apply \eqref{inter} onto the function $u:=(m+1)^{\frac{\sigma}{2}}$ with $\sigma$ given by \eqref{ddsig} to get
\begin{equation}\label{sigma-est}
\begin{split}
&\int_{Q}(m+1)^{\frac{\sigma(d+2)}{d}}\dx \dt = \int_0^T \|(m+1)^{\frac{\sigma}{2}}\|_{\frac{2(d+2)}{d}}^{\frac{2(d+2)}{d}}\dt\\
&\le C\left(\sup_{t\in(0,T)}\|(m(t)+1)^{\frac{\sigma}{2}} \|_2^2 + \int_0^T \|\nabla (m+1)^{\frac{\sigma}{2}} \|_2^2\dt\right)^{\frac{d+2}{d}}\\
&\le C\left(\sup_{t\in(0,T)}\|(m(t)+1) \|_{\sigma}^{\sigma} + \int_Q (m+1)^{\sigma-2}|\nabla m|^2\dx \dt \right)^{\frac{d+2}{d}}\\
&= C E^{\frac{d+2}{d}}.
\end{split}
\end{equation}
Finally, we also use the a~priori estimate \eqref{3.1} and the interpolation inequality
$$
\|\cdot\|_{q} \le \|\cdot\|_1^{1-\frac{q-1}{q} \frac{\sigma(d+2)}{\sigma(d+2)-d}} \|\cdot \|^{\frac{q-1}{q} \frac{\sigma(d+2)}{\sigma(d+2)-d}}_{\frac{\sigma(d+2)}{d}},
$$
which is valid for all $1\le q\le \sigma(d+2)/d$, to obtain
\begin{equation}\label{est-q}
\begin{split}
\int_Q (m+1)^q \dx \dt &\le \|m+1\|_{L^1(Q)}^{q-\frac{\sigma(d+2)(q-1)}{\sigma(d+2)-d}} \|m+1 \|^{\frac{\sigma(d+2)(q-1)}{\sigma(d+2)-d}}_{L^{\frac{\sigma(d+2)}{d}}(Q)}\\
&\le C\left(\int_Q (m+1)^{\frac{\sigma(d+2)}{d}}\dx \dt \right)^{\frac{d(q-1)}{\sigma(d+2)-d}}\\
&\le E^{\frac{(q-1)(d+2)}{\sigma(d+2)-d}},
\end{split}
\end{equation}
where for the last inequality we used \eqref{sigma-est}.

Next, we use \eqref{est-q} to handle the right hand side of \eqref{3.35}. Assuming that
\begin{equation}\label{10.1}
2s_0+1\le \frac{\sigma(d+2)}{d}
\end{equation}
we get from \eqref{est-q}
\begin{equation}\label{10.2}
\int_{Q}(m+1)^{2s_0+1}\dx \dt \le C E^{\frac{2s_0(d+2)}{\sigma(d+2)-d}}\le C\left(1+ \frac14 E\right),
\end{equation}
provided that
\begin{equation}\label{JF1}
\frac{2s_0 (d+2)}{\sigma(d+2)-d} <1,
\end{equation}
which by using of \eqref{ddsig} can be shown to be equivalent to \eqref{restr1}.
Notice that \eqref{JF1} directly implies the validity of \eqref{10.1}.
Hence, we can absorb the first term on the right hand side of \eqref{3.35} to get
\begin{equation}\label{3.355}
\begin{split}
E\le C\left(1+\|\bu\|_{L^{\infty}(Q)}^2 \left(1+ \int_{Q} (m+1)^{2s_0-\sigma +2}\dx \dt\right)\right)
\end{split}
\end{equation}
In case that
\begin{equation}\label{nic}
2s_0-\sigma +2 \le 1 \quad \Longleftrightarrow \quad 2(s_0 +s) \le r,
\end{equation}
we can use \eqref{3.1} to conclude \eqref{3.F2} directly.
If \eqref{nic} is not true, we again use \eqref{est-q} to get from \eqref{3.355}
\begin{equation*}
\begin{split}
E\le C\left(1+\|\bu\|_{L^{\infty}(Q)}^2 E^{\frac{(2s_0-\sigma +1)(d+2)}{\sigma(d+2)-d}}\right)
\end{split}
\end{equation*}
which after using the Young inequality leads to \eqref{3.F2},
provided that
\begin{equation}\label{nic2}
\frac{(2s_0-\sigma +1)(d+2)}{\sigma(d+2)-d}< 1.
\end{equation}
Note that \eqref{nic2} is a stronger assumption that \eqref{nic} and it can be shown by using \eqref{ddsig} that \eqref{nic2} is equivalent to \eqref{restr2}. Hence the proof of \eqref{3.F2} is complete.
\end{proof}

\subsection{Estimates for $\bu$} This subsection is devoted to the uniform estimates on $\bu$, which will still depend on $m$. To be more precise, we want to show that for arbitrary $p>d+2$ we have the estimate
\begin{equation}\label{UUF}
\|\bu\|_{L^{\infty}(Q)}\le  C+C(p)\|m^{2s_0}\|^2_{L^p(Q)}.
\end{equation}
\begin{proof}[Proof of \eqref{UUF}]
We start with the estimates for below for the quantity
$$
w:=\sum_{i=1}^N u^i.
$$
It is not difficult to observe from \eqref{1.6}, \eqref{Le1} and the fact that $2s\le r$, that $w$ satisfies almost everywhere in $Q$
\begin{equation}\label{u1}
-\partial_t w - \Delta w = \sum_{i=1}^N L^i (\bv,m,\nabla \bu) \ge -C\left(|\nabla w|^2  + m^{2s_0} +1\right).
\end{equation}
Next, let us consider $w_1$ a solution to
\begin{equation}\label{w1}
-\partial_t w_1 - \Delta w_1 = -Cm^{2s_0}
\end{equation}
completed by zero initial condition, i.e., $w_1(T)=0$. Then by a standard parabolic estimate, we obtain that
\begin{equation}
\|w_1\|_{L^{\infty}(Q)} + \|\nabla w_1\|_{L^{\infty}(Q)} \le C(p)\|m^{2s_0}\|_{L^p(Q)}
\label{est-w1}
\end{equation}
whenever $p>d+2$. Then subtracting \eqref{w1} from \eqref{u1}, we obtain
\begin{equation}\label{u-w1}
\begin{split}
-\partial_t &(w-w_1) - \Delta (w-w_1)  \ge -C\left(|\nabla w|^2+1\right)\\
&\ge -C|\nabla (w-w_1)|^2 -C(1+\|\nabla w_1\|_{\infty}^2).
\end{split}
\end{equation}
Hence from the theory for subsolutions to parabolic equation, see \cite{LaUr68}, we obtain
$$
w-w_1\ge - C(T)\max \{\|w(T)\|_{\infty}, (1+\|\nabla w_1\|^2_{L^{\infty}(Q)})\},
$$
which together with \eqref{est-w1} and the assumption that $\bu(T)\in L^{\infty}$ leads to the final estimate from below
\begin{equation}\label{Wfinal}
w\ge -C-C(p)\|m^{2s_0}\|^2_{L^p(Q)},
\end{equation}
which is valid for arbitrary $p>d+2$.

Next, we focus on estimates from above. Keeping the notation for $w$, we can derive from \eqref{1.6} that
$$
-\partial_t(u^i-\varepsilon_0w) - \Delta (u^i-\varepsilon_0 w) = L^i(m,\bv,\nabla \bu) - \varepsilon_0 \sum_{j=1}^N L^j(m,\bv,\nabla \bu).
$$
Hence, using \eqref{Le2}, we get
\begin{equation}
\begin{split}
&-\partial_t(u^i-\varepsilon_0w) - \Delta (u^i-\varepsilon_0 w)\le
C(1+m^{2s_0})\\
&\quad + C \left(1+\frac{m^{2s}+1}{m^r+1}\right)\left|\nabla\left(u^i- \varepsilon_0 w\right)\right|^2\\
&\le C(1+m^{2s_0})+ C \left|\nabla\left(u^i- \varepsilon_0 w\right)\right|^2,
\end{split}\label{subsol}
\end{equation}
where the second inequality follows from the assumption $2s\le r$.  Thus, we can repeat step by step the procedure for $w$ and using the fact that $\bu(T)\in L^{\infty}(\Omega)$, we obtain
\begin{equation}\label{Ufinal}
u^i-\varepsilon_0 w\le  C+C(p)\|m^{2s_0}\|^2_{L^p(Q)}.
\end{equation}

Finally, we derive the uniform bound \eqref{UUF}. First, summing \eqref{Ufinal} over $i=1,\ldots, N$ we have
\begin{equation*}
(1-N\varepsilon_0) w\le  CN+C(p)N\|m^{2s_0}\|^2_{L^p(Q)}.
\end{equation*}
Since $\varepsilon_0<1/(2N)$, we can combine this estimate with \eqref{Wfinal} to get
\begin{equation}\label{modW}
|w|\le  C+C(p)\|m^{2s_0}\|^2_{L^p(Q)}.
\end{equation}
Consequently, it follows from  \eqref{Ufinal} that
\begin{equation}\label{Ufinal2}
u^i\le  C+C(p)\|m^{2s_0}\|^2_{L^p(Q)}.
\end{equation}
Finally to obtain also estimate from below for $u^i$, we use \eqref{modW} and \eqref{Ufinal2} and get
$$
u^i=w-\sum_{j\neq i}u^j \overset{\eqref{modW},\eqref{Ufinal2}}\ge   C+C(p)\|m^{2s_0}\|^2_{L^p(Q)},
$$
which together with \eqref{Ufinal2} implies the desired estimate \eqref{UUF}. Hence the proof is complete.
\end{proof}

\subsection{Uniform $L^{\infty}$ bounds}
This subsection is devoted to the uniform bound for $\bu$, which directly implies the part of uniform estimates stated in \eqref{T-est}. Here, we combine \eqref{est-q}, \eqref{3.F2} and \eqref{UUF} to obtain the final bound. We go back to  \eqref{UUF} and estimate the right hand side. Although we need to choose $p>d+2$,  we formally provide all computation for $p=d+2$ and in the final restriction on the size of $s_0$ (or $\sigma$) we just use the strict inequality sign. Hence, we need to estimate the term on the right hand side of \eqref{UUF}, i.e., the integral
\begin{equation}
\begin{split}
\|m^{2s_0}\|_{d+2}^{2} = \left(\int_{Q}m^{2s_0(d+2)}\dx \dt\right)^{\frac{2}{d+2}}.
\end{split}
\end{equation}
If $2s_0(d+2) <1$ then the integral on the right hand side of \eqref{UUF} is bounded due to \eqref{3.1} and therefore we immediately get
\begin{equation}
\|\bu\|_{L^{\infty}(Q)}\le C(\|\bu_0\|_{\infty}, \|m_0\|_{\infty}).\label{infty-1}
\end{equation}
Hence, in what follows, we assume the opposite case. Assuming that
\begin{equation}\label{res1}
2s_0(d+2)<\sigma(d+2)/d \Leftrightarrow r+1 > 2ds_0+2s,
\end{equation}
we can use \eqref{est-q} with $q:=2s_0(d+2)$ and we deduce
\begin{equation}\label{est-qU}
\begin{split}
\left(\int_Q (m+1)^{2s_0(d+2)} \dx \dt\right)^{\frac{2}{d+2}} &\le C E^{\frac{(2s_0(d+2)-1)_+2}{\sigma(d+2)-d}}
\end{split}
\end{equation}
and it follows from \eqref{UUF} that
\begin{equation}\label{est-qUu}
\begin{split}
\|\bu\|_{L^{\infty}(Q)}&\le C \left(1+E^{\frac{(2s_0(d+2)-1)_+2}{\sigma(d+2)-d}}\right).
\end{split}
\end{equation}
Inserting this estimate into the right hand side of \eqref{3.F2}, we also deduce
\begin{equation}\label{est-qUuu}
\begin{split}
\|\bu\|_{L^{\infty}(Q)}&\le C \left(1+\|\bu\|_{L^{\infty}(Q)}^{\frac{4(2s_0(d+2)-1)_+}{\sigma(d+2)-d-(2s_0-\sigma+1)_+(d+2)}}\right).
\end{split}
\end{equation}
Hence, in case
\begin{equation}\label{incase}
\frac{4(2s_0(d+2)-1)_+}{\sigma(d+2)-d-(2s_0-\sigma +1)_+(d+2)} <1
\end{equation}
we can absorb the right hand side by left hand side and to obtain
\begin{equation}\label{est-qUuu2}
\begin{split}
\|\bu\|_{L^{\infty}(Q)}&\le C (\|\bu(T)\|_{\infty}, \|m_0\|_{\infty}),
\end{split}
\end{equation}
which implies a part of \eqref{T-est}. Notice that \eqref{incase} is a stronger assumption than \eqref{nic2} and therefore all needed assumptions, i.e., the assumptions \eqref{JF1} and \eqref{incase}, are already encoded in \eqref{ffres}. Furthermore, using \eqref{3.35} and \eqref{3.F2}, we obtain also the bound for $m$ and $\nabla m$ stated in \eqref{T-est}. Finally, from \eqref{A4}, we deduce the bound for term with $m(m^r+1)|\bv|^2$ in \eqref{T-est}.

\subsection{Uniform estimates for $\nabla \bu$}
This subsection is devoted to the last remaining part of \eqref{T-est}, i.e., the part of the estimate with $\nabla \bu$. Keeping the notation from the previous sections, we start with estimates for $\nabla w$. Using \eqref{Le1} andt the fact that $2s\le r$ we have
\begin{equation}\label{ww1}
-\partial_t w - \Delta w \ge (m^r+1)|\bv|^2 -C\left(|\nabla w|^2  + m^{2s_0} +1\right).
\end{equation}
Next, we multiply \eqref{ww1} by $e^{-2Cw}\ge 0$, integrate over $Q$ and use integration by parts to obtain
\begin{equation}\label{ww10}
\begin{split}
\int_Q &(m^r+1)|\bv|^2e^{-2Cw}  +2Ce^{-2Cw}|\nabla  w|^2 \dx \dt \\
&\le \frac{1}{2C} \int_Q\partial_t  e^{-2Cw}  +C\left(|\nabla w|^2  + m^{2s_0} +1\right)e^{-2Cw}\dx \dt.
\end{split}
\end{equation}
Thus, we see that we can absorb the term with $\nabla w$ by the left hand side and due to the $L^{\infty}$ bound for $\bu$, see \eqref{est-qUuu2}, and $L^{2s_0(d+2)}$ bound for $m$, see \eqref{est-qU}, we deduce from \eqref{ww10} that
\begin{equation}\label{ww11}
\begin{split}
\int_Q &(m^r+1)|\bv|^2 + |\nabla  w|^2 \dx \dt \le C(\|\bu_T\|_{\infty}, \|m_0\|_{\sigma}).
\end{split}
\end{equation}
Notice that the first in \eqref{ww11} together with \eqref{A4} leads to the estimate \eqref{T-est} for term involving $|\bv|^2$.

Next, we use the inequality \eqref{subsol}, i.e.,
\begin{equation}
\begin{split}
&-\partial_t(u^i-\varepsilon_0w) - \Delta (u^i-\varepsilon_0 w)\le C(1+m^{2s_0})+ C \left|\nabla\left(u^i- \varepsilon_0 w\right)\right|^2,
\end{split}\label{subsola}
\end{equation}
which we multiply by $e^{2C(u^i-\varepsilon_0w)}$ and integrate over $Q$. Repeating step by step the procedure \eqref{ww10}--\eqref{ww11} and using uniform bounds on $\bu$ and $m$, we get
\begin{equation}
\int_Q |\nabla (u^i-\varepsilon_0w)|^2 \dx \dt \le C(\|\bu_T\|_{\infty}, \|m_0\|_{\sigma}). \label{nablaq}
\end{equation}
Finally, combining \eqref{ww11} and \eqref{nablaq}, we have for all $i=1,\ldots,N$
$$
\begin{aligned}
\int_{Q}|\nabla u^i|^2 \dx\dt &\le 2\int_{Q} |\nabla (u^i-\varepsilon_0w)|^2 + \varepsilon_0^2|\nabla w|^2\dx \dt \\
&\le C(\|\bu_T\|_{\infty}, \|m_0\|_{\sigma}),
\end{aligned}
$$
which finishes the proof of \eqref{T-est}.

\section{Existence of solution - Proof of Theorem~\ref{T2}}\label{S.ex}
This section is devoted to the proof of the existence of a solution to \eqref{1.2}, \eqref{1.6} and \eqref{1.81}. Notice that due to the assumption \eqref{1.11} and \eqref{1.10}, we know that \eqref{1.10} is equivalent to
$$
\bv=\bw(m,\nabla \bu)
$$
with a Carath\'{e}odory mapping $\bw$. Therefore we can omit \eqref{1.81} and replace $\bv$ by $\bw(m,\nabla \bu)$ in \eqref{1.2} and \eqref{1.6} and solve the problem only for unknowns $(m,\bu)$. In fact, this is also the way how one can get the existence of a solution to an approximative problem. Nevertheless, for sake of simplicity and to simplify the notation, we keep writing $\bv$ in what follows.

Second, we do not provide the complete and rigorous proof here. We rather emphasize those steps that are different from the known procedure for Bellman systems. To be more specific, we provide here the proof of weak sequential stability, which is the key property of the system of equations we have in mind. It means that we shall consider a sequence of $(m^n,\bu^n, \bv^n)$ of smooth solutions to \eqref{1.2}, \eqref{1.6} and \eqref{1.81} (which is however equivalent to \eqref{ff}, once the mapping $\bw$ is well defined) and corresponding sequence of initial data
\begin{equation}\label{init-n}
\begin{aligned}
\bu_T^n &\to \bu_T && \textrm{strongly in } L^1(\Omega; \mathbb{R}^N),\\
\bu_T^n &\rightharpoonup^* \bu_T && \textrm{weakly$^*$ in } L^{\infty}(\Omega; \mathbb{R}^N),\\
m_0^n &\to m_0 && \textrm{stongly in } L^{\sigma}(\Omega)
\end{aligned}
\end{equation}
with nonnegative $m_0^n$. Our goal is to show that
\begin{equation}
(m^n, \bu^n, \bv^n) \to (m,\bu,\bv), \label{goal}
\end{equation}
strongly in  $L^1(Q)\times L^1(0,T; W^{1,2}(\Omega;\mathbb{R}^N)) \times L^1(Q;\mathbb{R}^{NM})$,
where the triple $(m,\bu,\bv)$ solves again \eqref{1.2}, \eqref{1.6} and \eqref{1.81} with initial data $(\bu_T, m_0)$. Indeed such a result the suggest that the rigorous existence proof is doable. Indeed, approximating Hamiltonians $H$ by a sequence of bounded functions $\{H^n\}_{n=1}^{\infty}$ and similarly $\bg$ by a sequence of bounded $\{\bg^n\}_{n=1}^{\infty}$, one may consider that for a such approximative system it is classical to obtain the existence of solution $(m^n,\bu^n,\bv^n)$ and the only remaining part of the proof is then the weak sequential stability. For details, how one can approximate Hamiltonians properly, we refer to \cite{BeBuFr12}.

\subsection{Uniform a~priori estimates}

I this part we just use the result of Theorem~\ref{T1}, which holds for sufficiently smooth solutions. Indeed, we may assume that
\begin{equation}\label{T-est2}
\begin{split}
&\sup_{t\in (0,T)} \left(\|m^n(t)\|_{\sigma} + \|\bu^n(t)\|_{\infty}\right)+  \int_{Q}|\nabla \bu^n|^2 + (m^n+1)^{\sigma -2}|\nabla m^n|^2\dx \dt \\
&\quad + \int_Q((m^n)^{r+1}+1)|\bv^n|^2  + (m^n)^{2s_0(d+2)} \dx \dt \\
&\qquad \le C(\|\bu^n_T\|_{\infty}, \|m^n_0\|_{\sigma})\le C
\end{split}
\end{equation}
such that $(m^n,\bu^n,\bv^n)$ satisfies for all $\varphi \in \mathcal{C}^{\infty}_0(-\infty;T; W^{1,\infty}_{per}(\Omega))$
\begin{equation}
\label{1.2P}
\int_Q -m^n \partial_t \varphi + \nabla  m^n \cdot \nabla \varphi -  m^n \bg(\bv^n, m^n)\cdot \nabla \varphi \dx \dt =\int_{\Omega}m_0^n \varphi(0)\dx,
\end{equation}
for all $\varphi \in \mathcal{C}^{\infty}_0(0; \infty; W^{1,\infty}(\Omega))$
\begin{equation}\label{1.6P}
\begin{split}
&\int_Q\bu^n \partial_t \varphi + \nabla \bu^n \cdot \nabla \varphi \dx \dt -\int_{\Omega} \bu^n_T \varphi(T)\dx \\
&\qquad = \int_Q L(m^n,\bv^n, \nabla \bu^n)\varphi\dx \dt
\end{split}
\end{equation}
and almost everywhere in $Q$
\begin{equation}
\label{1.81P}
f^i_{v^i_{j}}(\bv^n,m^n)  + b_1(m^n) \sum_{k=1}^d \partial_{x_k} (u^n)^i A^i_{kj} =0  \qquad \textrm{ in } Q
\end{equation}
with $\bg$ given as
\begin{equation}\label{g-defP}
\bg(\bv^n,m^n) := \sum_{j=1}^N b_1(m^n)A^j(v^n)^j + \bb_0(m^n).
\end{equation}
and $L(m^n,\bv^n,\bu^n)$ given as
\begin{equation}\label{Pp1}
\begin{split}
L(m^n,\bv^n, \nabla \bu^n)&:= \bef(\bv^n,m^n) + m^n\bef_m(\bv^n,m^n)\\
&\qquad + \nabla \bu^n \left[ \bg(\bv^n,m^n) + m^n\bg_{m^n} (\bv^n,m^n)\right].
\end{split}
\end{equation}

Next, we focus on the estimate for the time derivatives. First, using \eqref{Le3} and uniform bounds \eqref{T-est}, we see that
\begin{equation}
\int_{Q}|L(m^n,\bv^n, \nabla \bu^n)|\dx \; dt \le C. \label{L1L}
\end{equation}
Consequently, we can deduce from \eqref{1.6P} and also from \eqref{T-est2} that for some $q>d$
\begin{equation}
\int_0^T \|\partial_t \bu\|_{(W^{1,q}_{per}(\Omega; \mathbb{R}^N))^*} \dt \le C.\label{ut}
\end{equation}
Similarly, using \eqref{2.6} and \eqref{T-est2}, we see that for some $q\in (1,\infty)$
\begin{equation}
\int_0^T \|\partial_t m\|^{q'}_{(W^{1,q}_{per}(\Omega))^*} \dt \le C.\label{mt}
\end{equation}

\subsection{Limit $n\to \infty$}
Having \eqref{T-est2}, \eqref{init-n} and \eqref{ut}--\eqref{mt} and using the Aubin--Lions lemma, we can find subsequences that we do not relabel, such that for some $q>d$
\begin{align}
\bu^n &\rightharpoonup \bu && \textrm{weakly in } L^2(0,T; W^{1,2}_{per}(\Omega; \mathbb{R}^N)), \label{cn1}\\
\bu^n &\rightharpoonup^* \bu && \textrm{weakly$^*$ in } L^{\infty}(Q; \mathbb{R}^N), \label{cn2}\\
\bu^n &\to \bu && \textrm{strongly in } L^{1}(Q; \mathbb{R}^N), \label{cn2.5}\\
\partial_t \bu^n &\rightharpoonup^* \partial_t \bu && \textrm{weakly$^*$ in } \mathcal{M}(0,T; (W^{1,q}_{per}(\Omega; \mathbb{R}^N))^*), \label{cn3}\\
m^n &\rightharpoonup  m && \textrm{weakly in } L^{2}(0,T; W^{1,2}_{per}(\Omega)), \label{cn4}\\
m^n &\to  m && \textrm{strongly in } L^{1}(Q), \label{cn4.5}\\
m^n &\rightharpoonup  m && \textrm{weakly in } L^{\frac{\sigma(d+2)}{d}}(Q), \label{cn5}\\
\partial_t m^n &\rightharpoonup \partial_t m && \textrm{weakly in } L^{q'}(0,T; (W^{1,q}_{per}(\Omega)^*), \label{cn7}\\
L(m^n,\bv^n, \nabla \bu^n) &\rightharpoonup^* L && \textrm{weakly$^*$ in } \mathcal{M}(Q), \label{cn8}
\end{align}
where $\mathcal{M}(K)$ denotes the space of Radon measures on $K$. Having \eqref{cn1}--\eqref{cn8}, we can use the theory for parabolic equations with $L^1$ or measure right hand side, see \cite{BoGa89,BoDaGaOr97}, and to conclude that
\begin{equation}\label{cn9}
\nabla \bu^n \to \nabla \bu \quad \textrm{ a.e. in }Q.
\end{equation}
Consequently, since the operator $T$ in \eqref{1.11}--\eqref{1.10} is strictly monotone, it follows from \eqref{cn4.5} and \eqref{cn9} that
\begin{equation}\label{cn10}
\bv^n \to \bv \quad \textrm{ a.e. in }Q,
\end{equation}
which in particular also implies (note that $L$ and $\bg$ are Carath\'{e}odory)
\begin{align}
L(m^n,\bv^n, \nabla \bu^n) &\to L(m,\bv, \nabla \bu) &&\textrm{ a.e. in }Q,\label{cn11}\\
\bg(\bv^n, m^n) &\to \bg(\bv, m) &&\textrm{ a.e. in }Q. \label{cn12}
\end{align}
Thus, the convergence results \eqref{cn1}--\eqref{cn12} allows us to let $n\to \infty$ in \eqref{1.2P}--\eqref{g-defP} to obtain that $(m,\bv,\bu)$ solves \eqref{wf-m}, \eqref{wf-m} and
\begin{equation}\label{1.6PF}
\begin{split}
&\int_Q\bu \partial_t \varphi + \nabla \bu \cdot \nabla \varphi \dx \dt -\int_{\Omega} \bu_T \varphi(T)\dx  = \int_Q L\varphi\dx \dt
\end{split}
\end{equation}
for all $\varphi \in \mathcal{C}^{\infty}_0(0,\infty; W^{1,\infty}(\Omega))$. Thus, to finish the proof of Theorem~\ref{T2}, we need to show that
\begin{equation}
L=L(m,\bv, \nabla \bu) \textrm{ a.e. in }Q. \label{finish}
\end{equation}
Indeed, having \eqref{finish}, we can read from the equation that the time derivative of $\bu$ is better, i.e., that
$$
\bu \in L^1(0,T; (W^{1,q}_{per}(\Omega; \mathbb{R}^N))^*)
$$
for some $q>d$ and integrating by parts with respect to the time variable $t$ in \eqref{1.6PF}, we find that \eqref{wf-u} holds true.

\subsection{Identification of $L$ - proof of \eqref{finish}}
This last part is devoted to the proof of \eqref{finish}. We follow the procedure developed in \cite{BeBuFr12} with the necessary changes due to the presence of the mean field variable $m$. We also proceed here slightly formally ad refer the interested reader to \cite{BeBuFr12}, where the very similar  procedure is made rigorously. Defining
$$
w^n:=\sum_{i=1}^N (u^n)^i,
$$
it follows from \eqref{1.6P} that in the sense of distributions we have
\begin{equation}\label{1.6Pwn}
\begin{split}
&-\partial_t w^n - \Delta w^n = \sum_{i=1}^N  L(m^n,\bv^n, \nabla \bu^n).
\end{split}
\end{equation}
Next, we multiply \eqref{1.6Pwn} by
$$
e^{-2Cw^n}
$$
and obtain
\begin{equation}\label{1.6Pwne}
\begin{split}
&\frac{1}{2C}\partial_t e^{-2Cw^n} + \frac{1}{2C}\Delta e^{-2Cw^n} \\
&\quad = \sum_{i=1}^N  L(m^n,\bv^n, \nabla \bu^n)e^{-2Cw^n} + 2Ce^{-2Cw^n}|\nabla w^n|^2.
\end{split}
\end{equation}
Consequently, multiplying the resulting identity by arbitrary nonnegative  $\varphi\in W^{1,1}_0(0,T: L^{\infty}(\Omega)\cap W^{1,2}_{per}(\Omega))$ and integrating the result over $Q$, we get by using integration by parts that
\begin{equation}\label{1.6Pwnei}
\begin{split}
&-\frac{1}{2C}\int_{Q} e^{-2Cw^n}\partial_t \varphi + \nabla  e^{-2Cw^n}\cdot \nabla \varphi \dx \dt \\
&\quad = \int_Q \varphi(\sum_{i=1}^N  L(m^n,\bv^n, \nabla \bu^n)e^{-2Cw^n} + 2Ce^{-2Cw^n}|\nabla w^n|^2).
\end{split}
\end{equation}
Next, we let $n\to \infty$. Thanks to \eqref{cn1}, \eqref{cn2} and \eqref{cn2.5}, it is easy to deduce that
\begin{equation}\label{1.6Pwneij}
\begin{split}
&\lim_{n\to \infty} \int_{Q} e^{-2Cw^n}\partial_t \varphi + \nabla  e^{-2Cw^n}\cdot \nabla \varphi \dx \dt \\
& = \int_{Q} e^{-2Cw}\partial_t \varphi + \nabla  e^{-2Cw}\cdot \nabla \varphi \dx \dt
\end{split}
\end{equation}
Next, using \eqref{Le1}, we see that
$$
\sum_{i=1}^N  L(m^n,\bv^n, \nabla \bu^n)e^{-2Cw^n} + 2Ce^{-2Cw^n}|\nabla w^n|^2 \ge -C ((m^n)^{2s_0}+1).
$$
Consequently, we see that the right hand side of \eqref{1.6Pwnei} is bounded from below by a strongly convergent function, so using the point-wise convergence result \eqref{cn11} and the Fatou lemma, we get
\begin{equation}\label{1.6Pwnep}
\begin{split}
&\liminf_{n\to \infty} \int_Q \varphi(\sum_{i=1}^N  L(m^n,\bv^n, \nabla \bu^n)e^{-2Cw^n} + 2Ce^{-2Cw^n}|\nabla w^n|^2)\\
&\ge \quad \int_Q \varphi(\sum_{i=1}^N  L(m,\bv, \nabla \bu)e^{-2Cw} + 2Ce^{-2Cw}|\nabla w|^2).
\end{split}
\end{equation}
Hence, combining \eqref{1.6Pwneij} and \eqref{1.6Pwnep}, we see that for all nonnegative $\varphi\in W^{1,1}_0(0,T: L^{\infty}(\Omega)\cap W^{1,2}_{per}(\Omega))$ there holds
\begin{equation}\label{1.6nee}
\begin{split}
&-\frac{1}{2C}\int_{Q} e^{-2Cw}\partial_t \varphi + \nabla  e^{-2Cw}\cdot \nabla \varphi \dx \dt \\
&\quad \ge \int_Q \varphi(\sum_{i=1}^N  L(m,\bv, \nabla \bu)e^{-2Cw} + 2Ce^{-2Cw}|\nabla w|^2).
\end{split}
\end{equation}
Finally, taking $\varphi:=e^{2Cw}\psi$ with arbitrary $\psi \in W^{1,1}_0(0,T; L^{\infty}(\Omega)\cap W^{1,2}_{per}(\Omega))$, we deduce that\footnote{In fact we must mollify the test function with respect to the time variable and then to pas to the limit. Since such a procedure was explained in details in \cite{BeBuFr12}, we do not provide the complete proof here.}
\begin{equation}\label{okn}
\begin{split}
&\int_{Q} w\partial_t \psi  + \nabla w \cdot \psi \dx \dt \ge \int_Q \psi(\sum_{i=1}^N  L(m,\bv, \nabla \bu))\psi \dx \dt,
\end{split}
\end{equation}
which implies that in the sense of distributions
\begin{equation}\label{oknn}
\begin{split}
&-\partial_t w - \Delta w \ge  \sum_{i=1}^N  L(m,\bv, \nabla \bu).
\end{split}
\end{equation}

Similarly, we deduce the opposite type inequalities. It follows from \eqref{1.6P} that for all $i=1,\ldots,N$
\begin{equation}\label{1.6Ps}
\begin{split}
&-\partial_t((u^i)^n - \varepsilon_0 w^n)  -\Delta ((u^i)^n - \varepsilon_0 w^n) \\
&\qquad = L^i(m^n,\bv^n, \nabla \bu^n)-\varepsilon_0\sum_{i=1}^N L^i(m^n,\bv^n, \nabla \bu^n).
\end{split}
\end{equation}
Denoting $z^n:=(u^i)^n - \varepsilon_0 w^n$ ad multiplying \eqref{1.6Ps} by $e^{2Cz^n}$ we get
\begin{equation}\label{1.6Pss}
\begin{split}
&-\frac{1}{2C}\partial_te^{2Cz^n}  -\frac{1}{2C}\Delta e^{2Cz^n} \\
&= e^{2Cz^n}\left(L^i(m^n,\bv^n, \nabla \bu^n)-\varepsilon_0\sum_{i=1}^N L^i(m^n,\bv^n, \nabla \bu^n)\right) -e^{2Cz^n}|\nabla z^n|^2.
\end{split}
\end{equation}
Hence, using \eqref{Le2}, we see that the right hand side is bounded by an convergent sequence and therefore we can proceed similarly as before by using the Fatou lemma and to obtain
\begin{equation}\label{1.6Pok}
\begin{split}
&-\partial_t(u^i - \varepsilon_0 w)  -\Delta (u^i - \varepsilon_0 w) \\
&\qquad \le  L^i(m,\bv, \nabla \bu)-\varepsilon_0\sum_{i=1}^N L^i(m,\bv, \nabla \bu).
\end{split}
\end{equation}
Thus summing with respect to $i=1,\ldots,N$ and dividing by $(1-\varepsilon_0)$ we see that
\begin{equation*}
\begin{split}
&-\partial_t w  -\Delta w \\
&\qquad \le \sum_{i=1}^N L^i(m,\bv, \nabla \bu),
\end{split}
\end{equation*}
which combined with \eqref{oknn} gives
\begin{equation}\label{1.6huhu}
\begin{split}
&-\partial_t w  -\Delta w = \sum_{i=1}^N L^i(m,\bv, \nabla \bu)
\end{split}
\end{equation}
and consequently also we obtain from \eqref{1.6Pok} that
\begin{equation}\label{1.6say}
\begin{split}
&-\partial_t u^i   -\Delta  u^i \le  L^i(m,\bv, \nabla \bu).
\end{split}
\end{equation}
Finally using \eqref{1.6huhu} and \eqref{1.6say}, we get
\begin{align*}
-\partial_t u^i   -\Delta  u^i &= -\partial_t w   -\Delta  w + \sum_{j\neq i}(\partial_t u^j   +\Delta  u^j)\\
&\ge \sum_{j=1}^N L^j(m,\bv, \nabla \bu) - \sum_{j\neq i} L^j(m,\bv, \nabla \bu)\\
&=L^i(m,\bv, \nabla \bu).
\end{align*}
Hence, \eqref{1.6say} implies that
\begin{equation}\label{1.6say2}
\begin{split}
&-\partial_t u^i   -\Delta  u^i =  L^i(m,\bv, \nabla \bu)
\end{split}
\end{equation}
and \eqref{wf-u} follows. This completes the proof of Theorem~\ref{T2}.

\def\cprime{$'$}
\providecommand{\bysame}{\leavevmode\hbox to3em{\hrulefill}\thinspace}
\providecommand{\MR}{\relax\ifhmode\unskip\space\fi MR }
\providecommand{\MRhref}[2]{%
  \href{http://www.ams.org/mathscinet-getitem?mr=#1}{#2}
}
\providecommand{\href}[2]{#2}

\end{document}